\RequirePackage{fix-cm}
\documentclass[smallextended]{svjour3}       
\smartqed  
\usepackage{graphicx}

\usepackage{amsmath}
\usepackage{amsfonts}
\usepackage{amssymb}
\usepackage{amscd}

\usepackage{hyperref}
\begin{document}

\title{ Generalizations of prime submodules over non-commutative rings
}


\author{Emel Aslankarayigit Ugurlu        
}


\institute{E. Aslankarayigit Ugurlu \at
             Department of Mathematics, Marmara University, Istanbul, Turkey. \\ \email{emel.aslankarayigit@marmara.edu.tr}           
     }

\date{Received: 01 March 2020 / Accepted: date}

\maketitle

\begin{abstract}
Throughout this paper, $R$ is an \textit{associative ring (not necessarily
	commutative) with identity} and $M$ is a right $R$-module with unitary. In
this paper, we introduce a new concept of $\phi$\textit{-prime submodule over
	an associative ring} \textit{with identity}. Thus we define the concept as
following: Assume that $S(M)$ is the set of all submodules of $M$ and
$\phi:S(M)\rightarrow S(M)\cup\{\emptyset\}$ is a function. For every $Y\in
S(M)$ and ideal $I$ of $R,$ a proper submodule $X$ of $M$ is called
\textit{$\phi$-prime,} if $YI\subseteq X$ and $YI\nsubseteq\phi(X),$ then
$Y\subseteq X$ or $I\subseteq(X:_{R}M)\mathit{.}$ Then we examine the
properties of \textit{$\phi$-}prime submodules and characterize it when $M$ is
a \textit{multiplication module.}

\keywords{ $\phi-$prime submodule \and non-commutative ring \and multiplication module.}
\subclass{16P40 \and 13A15 \and 16D60}
\end{abstract}

\section{Introduction}

\label{Sec:1}

Throughout this paper, $R$ is an associative ring (unless otherwise stated,
not necessarily commutative) with identity and $M$ is a right $R$-module with
unitary. Suppose that $M$ is an $R$-module, $S(M)$ and $S(R)$ are the set of
all submodules of $M$, the set of all ideals of $R$, respectively. For an
ideal $A$ of $R$, we denote the set $\{t\in M:$ $tA\subseteq X\}$ as
$(X:_{M}A).$ One clearly proves that $(X:_{M}A)\in S(M)$ and $X\subseteq$
$(X:_{M}A).$ Also, for two subsets $X$ and $Y$ of $M$, the subset $\{r\in
R:Xr\subseteq Y\}$ of $R$ is denoted by $(Y:_{R}X).$ If $Y$ is a submodule of
$M$, then it is obviously proved that for any subset $X$ of $M$, the set
$(Y:_{R}X)$ is a right ideal of $R$. It is obtained $(Y:_{R}X)$ is an ideal of
$R$ for $X$,$Y\in S(M)$, see \cite{tug}. Thus, clearly one can see that
$(X:_{R}M)$ is an ideal of $R,$ for all $X\in S(M).$

\bigskip

A proper ideal $A$ of a commutative ring $R$ is \textit{prime }if whenever
$a_{1},a_{2}\in R$ with $a_{1}a_{2}\in A$, then $a_{1}\in A$ or $a_{2}\in A,$
\cite{atiyah}. In 2003, the authors \cite{AS} said that if whenever
$a_{1},a_{2}\in R$ with $0_{R}\neq a_{1}a_{2}\in A$, then $a_{1}\in A$ or
$a_{2}\in A,$ a proper ideal $A$ of a commutative ring $R$ is \textit{weakly
	prime}. In \cite{BS}, Bhatwadekar and Sharma defined a proper ideal $A$ of an
integral domain $R$ as \textit{almost prime (resp. $n$-almost prime)} if for
$a_{1},a_{2}\in R$ with $a_{1}a_{2}\in A-A^{2}$, (resp. $a_{1}a_{2}\in
A-A^{n},n\geq3$) then $a_{1}\in A$ or $a_{2}\in A$. This definition can be
made for any commutative ring $R$. Later, Anderson and Batanieh \cite{AB}
introduced a concept which covers all the previous definitions in a
commutative ring $R$ as following: Let $\phi:S(R)\rightarrow S(R)\cup
\{\emptyset\}$ be a function. A proper ideal $A$ of a commutative ring $R$ is
called $\phi$-\textit{prime} if for $a_{1},a_{2}\in R$ with $a_{1}a_{2}\in
A-\phi(A)$, then $a_{1}\in A$ or $a_{2}\in A.$

\bigskip

The notion of the prime ideal in a commutative ring $R$ is extended to modules
by several studies, \cite{D,L,MM}. For a commutative ring $R$, a proper $X\in
S(M)$ is said to be \textit{prime }\cite{Am}, if $ma\in X,$ then $m\in X$ or
$a\in(X:_{R}M),$ for $a\in R$ and $m\in M.$ In \cite{AT}, the authors
introduced weakly prime submodules over a commutative ring $R$ as following: A
proper submodule $X$ of $M$ is called \textit{weakly prime} if for $r\in R$
and $m\in M$ with $0_{M}\neq mr\in X$, then $m\in X$ or $r\in(X:_{R}M)$. Then,
N. Zamani \cite{Z} introduced the concept of $\phi$-prime submodules over a
commutative ring $R$ as following: Let $\phi:S(M)\rightarrow S(M)\cup
\{\emptyset\}$ be a function. A proper submodule $X$ of an $R$-module $M$ is
said to be \textit{$\phi$-prime} if $r\in R$, $m\in M$ with $mr\in X-\phi(X)$,
then $m\in X$ or $r\in(X:_{R}M)$. He defined the map $\phi_{\alpha
}:S(M)\rightarrow S(M)\cup\{\emptyset\}$ as follows:

\begin{itemize}
	\item[(1)] $\phi_{\emptyset}$ : $\phi(X)=\emptyset$ defines prime submodules.
	
	\item[(2)] $\phi_{0}$ : $\phi(X)=\{0_{M}\}$ defines weakly prime submodules.
	
	\item[(3)] $\phi_{2}$ : $\phi(X)=X(X:_{R}M)$ defines almost prime submodules.
	
	\item[(4)] $\phi_{n}$ : $\phi(X)=X(X:_{R}M)^{n-1}$ defines $n$-almost prime
	submodules$(n\geq2).$
	
	\item[(5)] $\phi_{\omega}$ : $\phi(X)=\cap_{n=1}^{\infty}X(X:_{R}M)^{n}$
	defines $\omega$-prime submodules.
	
	\item[(6)] $\phi_{1}$ : $\phi(X)=X$ defines any submodule.
\end{itemize}

\bigskip

On the other hand, in \cite{Beiranvand}, P. Karimi Beiranvand and R.
Beyranvand introduced the almost prime and weakly prime submodules over $R$
(not necessarily commutative) as following: A proper submodule $X$ of an
$R$-module $M$ is called \textit{almost}$\mathit{\ }$\textit{prime,} for any
ideal $I$ of $R$ and any submodule $Y$ of $M,$ if $YI\subseteq X$ and
$YI\nsubseteq X(X:_{R}M),$ then $Y\subseteq X$ or $I\subseteq(X:_{R}M)$. Also,
$X$ is called \textit{weakly}$\mathit{\ }$\textit{prime,} for any ideal $I$ of
$R$ and any submodule $X$ of $M,$ if $0_{M}\neq YI\subseteq X$, then
$Y\subseteq X$ or $I\subseteq(X:_{R}M)$. In the mentioned study, they obtain
some important results on the two submodules over $R$.

\bigskip

In any \textit{non-commutative ring,} T. Y. Lam \cite{lam} proved that an
ideal $A$ of $R$ is a prime ideal (i.e., for two ideals $I_{1}$, $I_{2}$ of
$R$, $I_{1}I_{2}$ $\subseteq A$ implies $I_{1}\subseteq A$ or $I_{2}$
$\subseteq A$ ) $\Longleftrightarrow$ for $a_{1},a_{2}\in R,$ $a_{1}a_{2}\in
A$ implies $a_{1}\in A$ or $a_{2}\in A.$ Similarly, for any module over any
\textit{non-commutative ring,} J. Dauns \cite{D} showed that for $M$ over $R,$
a proper $X$ $\in S(M)$ is prime (i.e., if $mRa\subseteq X,$ then $m\in X$ or
$a\in(X:_{R}M),$ for $a\in R$ and $m\in M$) $\Longleftrightarrow$ for an ideal
$A$ of $R$ and for a submodule $Y$ of $M,$ $YA\subseteq X$ implies $Y\subseteq
X$ or $A\subseteq(X:_{R}M).$

\bigskip

Moreover, note that in commutative ring theory, we know that there is a
relation between prime ideals and multiplicatively closed sets. Similarly, in
\textit{non-commutative ring theory}, there is a relation between prime ideals
and $m$-\textit{system} sets. In \cite{lam}, one can see that if for all
$x,y\in S,$ there exists $a\in R$ with $xay\in S,$ then $\emptyset\neq
S\subseteq R$ is called an $m$-system$.$ Also, T. Y. Lam \cite{lam} defined
the radical of an ideal $A$ of $R$ as: $\sqrt{A}=\{s\in R:$ every $m$-system
containing $s$ meets $A\}\subseteq\{s\in R:$ $s^{n}\in A$ for some $n\geq1\}.$
Then he proved that $\sqrt{A}$ equals the intersection of all prime ideals
containing $A$ and $\sqrt{A}$ is an ideal$,$ see, (10.7) Theorem in \cite{lam}.

\bigskip

Our aim in this paper, similar to \cite{Beiranvand}, to introduce the concept
of $\phi$-prime submodule over an associative ring (not necessarily
commutative) with identity. For this purpose, we define a $\phi$-prime
submodules over $R$. In Section 2, after the introducing of $\phi$-prime
submodules over $R$, in Theorem \ref{the main}, we characterize a $\phi$-prime submodule. Then with Theorem \ref{the def}, we give another equivalent definitions for $\phi$-prime submodule. Also, in the section some properties of the submodules are examined. In Theorem \ref{the cha}, another characterization of  $\phi$-prime submodule is
obtained. In Section 3, after a reminder about multiplication module, it is
shown that $X$ is $\phi$-prime $\Longleftrightarrow$ $Y_{1}Y_{2}\subseteq X$
and $Y_{1}Y_{2}\nsubseteq\phi(X)$ implies $Y_{1}\subseteq X$ or $Y_{2}%
\subseteq X,$ for $Y_{1}$, $Y_{2}\in S(M)$, see Corollary \ref{cor multip}.
Moreover, in Theorem \ref{the ideal iff}, for a multiplication module, under some conditions we prove that 	$X$ is $\phi$-prime in $M$ $\Longleftrightarrow$ $(X:_{R}M)$ is a $\psi$-prime ideal in $R.$ In Section 4,
with Definition \ref{def system}, we introduce a new concept which is called
$\phi$-$m$-$system.$ Then we show that in Proposition \ref{system}, for $X\in
S(M)$, $X$ is $\phi$-prime $\Longleftrightarrow$ $S=M-X$ is a $\phi$%
-$m$-system. Also, we examine some properties of the $\phi$-$m$-$system.$
Finally, with Definition \ref{def radical}, we introduce \textit{the radical
	of} $Y$ as $\sqrt{Y}:=\{x\in M:$ every $\phi$-$m$-system $S$ containing $x$
such that $\phi(Y)=\phi(<S^{c}>)$ meets $Y\},$ otherwise $\sqrt{Y}:=M,$ where
$S^{c}=M-S.$ As a final result, for the set $\Omega:=\{X_{i}\in S(M):X_{i}$ is
$\phi$-prime with $Y\subseteq X_{i}$ and $\phi(Y)=\phi(X_{i}),$ for
$i\in\Lambda$ $\},$ it is obtained that $\sqrt{Y}=\underset{X_{i}\in\Omega
}{\bigcap}X_{i},$ see Theorem \ref{the radical}.

\section{Properties of $\phi-$Prime submodules}

\label{Sec:2} Throughout our study, assume that $\phi:S(M)\rightarrow
S(M)\cup\{\emptyset\}$ is a function.

\begin{definition}
	For every $Y\in S(M)$ and $I\in S(R),$ a proper $X\in S(M)$ is said to be
	\textit{$\phi$-prime,} if $YI\subseteq X$ and $YI\nsubseteq\phi(X),$ then
	$Y\subseteq X$ or $I\subseteq(X:_{R}M)$\textit{.} We defined the map
	$\phi_{\alpha}:S(M)\rightarrow S(M)\cup\{\emptyset\}$ as follows:
\end{definition}

\begin{itemize}

	\item[(1)] $\phi_{\emptyset}$ : $\phi(X)=\emptyset$ defines prime submodules.
	
	\item[(2)] $\phi_{0}$ : $\phi(X)=\{0_{M}\}$ defines weakly prime submodules.
	
	\item[(3)] $\phi_{2}$ : $\phi(X)=X(X:_{R}M)$ defines almost prime submodules.
	
	\item[(4)] $\phi_{n}$: $\phi(X)=X(X:_{R}M)^{n-1}$ defines $n$-almost prime
	submodules$(n\geq2)$.
	
	\item[(5)] $\phi_{\omega}$ : $\phi(X)=\cap_{n=1}^{\infty}X(X:_{R}M)^{n}$
	defines $\omega$-prime submodules.
	
	\item[(6)] $\phi_{1}$ : $\phi(X)=X$ defines any submodule.
\end{itemize}

\bigskip In the above definition, if we consider $\phi:S(R)\rightarrow
S(R)\cup\{\emptyset\}$, we obtain the concept of $\phi$-prime ideal in an
associative ring (not necessarily commutative) with identity as following: For
every $I,J\in S(R),$ a proper $A\in S(R)$ is said to be \textit{$\phi$-prime,}
if $IJ\subseteq A$ and $IJ\nsubseteq\phi(A),$ then $I\subseteq A$ or
$J\subseteq A$\textit{.} For commutative case, this definition is equivalent
to the definition of $\phi$-prime ideal in a commutative ring, see the Theorem
13 in \cite{AB}.

\bigskip Notice that since $X-\phi(X)=X-(X\cap\phi(X))$, for any submodule $X$
of $M$, without loss of generality, suppose $\phi(X)\subseteq X.$ Let
$\psi_{1},$ $\psi_{2}:S(M)\rightarrow S(M)\cup\{\emptyset\}$ be two functions,
if $\psi_{1}(X)\subseteq\psi_{2}(X)$ for each $X\in S(M),$ we denote $\psi
_{1}\leq\psi_{2}.$ Thus clearly, we have the following order: $\phi
_{\emptyset}\leq\phi_{0}\leq\phi_{\omega}\leq...\leq\phi_{n+1}\leq\phi_{n}%
\leq...\leq\phi_{2}\leq\phi_{1}$. Whenever $\psi_{1}\leq\psi_{2}$, any
$\psi_{1}$-prime submodule is $\psi_{2}$-prime.

\bigskip

\begin{example}
	Let $p$ and $q$ be two prime numbers. Consider $%
	\mathbb{Z}
	-$module $%
	\mathbb{Z}
	_{pq}.$ The zero submodule is $\phi_{0}-$prime, but it is not $\phi
	_{\emptyset}-$prime. Moreover, in $%
	\mathbb{Z}
	-$module $%
	\mathbb{Z}
	_{pq^{2}},$ the submodule $q^{2}%
	\mathbb{Z}
	_{pq^{2}}$ is $\phi_{2}-$prime. However, since $q^{2}%
	\mathbb{Z}
	_{pq^{2}}(q^{2}%
	\mathbb{Z}
	_{pq^{2}}:_{%
		\mathbb{Z}
	}%
	\mathbb{Z}
	_{pq^{2}})=q^{2}%
	\mathbb{Z}
	_{pq^{2}},$ it is not $\phi_{0}-$prime.
\end{example}

\begin{example}
	Let $M$ be an $R$-module.
	
	\begin{enumerate}
		\item The zero submodule of $R$ is both $\phi_{0}-$prime submodule and
		$\phi_{2}-$prime submodule, on the other hand it may not be $\phi_{\emptyset
		}-$prime.
		
		\item If $\ M$ is a prime $R$-module and $N$ be a proper submodule of $M$.
		Then $N$ is $\phi_{\emptyset}-$prime if and only if $\phi_{0}-$prime.
		
		\item Let $M$ be a homogeneous semisimple $R$-module and $N$ be a proper
		submodule of $M$. Then since every proper submodule is prime, hence $N$ is
		prime, so is $\phi-$prime.
	\end{enumerate}
\end{example}

\begin{example}
	(Example 2.2 (f) in \cite{Beiranvand})Let $M=S_{1}%
	{\textstyle\bigoplus}
	S_{2},$ which $S_{1},S_{2}$ are simple $R$-module such that $S_{1}\ncong
	S_{2}$ and $N$ be a proper submodule of $M.$ Then since every non-zero proper
	submodule is prime, then $N$ is prime, so is $\phi-$prime. Indeed, assume that
	$0_{M}\neq X\in S(M)$ is proper and $YI\subseteq X$ where $Y\in S(M)$ and
	$I\in S(R).$ By Proposition 9.4 in \cite{FW}, we have $M/X\cong S_{1}$ or
	$M/X\cong S_{2}.$ Then $((Y+X)/X)I=0_{M}$ and as $(Y+X)/X\in S(M/X)$ and $M/X$
	is simple, we get $(Y+X)/X=0_{M}$ or $Ann((Y+X)/X)=Ann(M/X).$ This means that
	$Y+X=X$ or $(M/X)I=0_{M}.$ Consequently, $Y\subseteq X$ or $MI\subseteq X.$
\end{example}

Note that for an element $a$ of $R,$\ the ideal generated by $a$ in $R$ is
denoted by $RaR.$ Similarly, the right and left ideal generated by $a$ in $R$
are denoted by $aR,$ $Ra,$ respectively. Also, we denote the ideal generated
by $A$\ as $<A>,$ for a subset $A$ of $R.$ For an element $x$ of $M,$\ the
submodule generated by $x$ in $M$ is denoted by $xR.$ Finally, for a subset
$X$ of $M$, we denote the submodule generated by $X$ in $M$ as $<X>$.

\bigskip In the following Theorem, we obtain a characterization of a $\phi
$-prime submodule of $M$.

\begin{theorem}
	\label{the main} For a proper submodule $X$ of $M$, the followings are equivalent:
	
	\begin{enumerate}
		\item $X$ is a $\phi$-prime submodule of $M.$
		
		\item For all $m\in M-X$,
		
		$(X:_{R}mR)=(X:_{R}M)\cup(\phi(X):_{R}mR).$
		
		\item For all $m\in M-X$,
		
		$(X:_{R}mR)=(X:_{R}M)$ or $(X:_{R}mR)=(\phi(X):_{R}mR).$
	\end{enumerate}
\end{theorem}

\begin{proof}
	$(1)\Longrightarrow(2)$ : Let $X$ be a $\phi$-prime submodule of $M.$ For all
	$m\in M-X$, choose $a\in(X:_{R}mR)-(\phi(X):_{R}mR).$ Then $(mR)(RaR)\subseteq
	X$ and $(mR)(RaR)\nsubseteq\phi(X).$ As $X$ is $\phi$-prime, one can see
	$mR\subseteq X$ or $RaR\subseteq(X:_{R}M).$ The first option gives us a
	contradiction. Thus $a\in(X:_{R}M).$ Moreover, as $\phi(X)\subseteq X,$ we
	always have $(\phi(X):_{R}mR)\subseteq(X:_{R}mR).$
	
	$(2)\Longrightarrow(3)$ : If an ideal is a union of two ideals, it equals to
	one of them.
	
	$(3)\Longrightarrow(1)$ : Choose $Y\in S(M)$\ and an ideal $I$ in $R$ which
	$YI\subseteq X$ and $I\nsubseteq(X:_{R}M)$, $Y\nsubseteq X.$ Let us prove
	$YI\subseteq\phi(X).$ For all $r\in I$ and $m\in Y,$ we have $mr\in
	YI\subseteq X.$ 
	
	Now, take $m\in Y-X.$ Then we have 2 cases:
	
	Case 1: $r\notin(X:_{R}M).$ Since $mr\in YI\subseteq X,$ one can see
	$(mR)r\subseteq YI\subseteq X,$ i.e., $r\in(X:_{R}mR).$ Thus $(X:_{R}%
	mR)=(\phi(X):_{R}mR)$ by our hypothesis (3). This means $r\in(\phi
	(X):_{R}mR),$ so, $mr\in\phi(X).$
	
	Case 2 : $r\in(X:_{R}M).$ Thus $r\in I\cap(X:_{R}M).$ Choose $s\in
	I-(X:_{R}M).$ Thus $r+s\in I-(X:_{R}M).$ Similar to Case 1, since
	$s\notin(X:_{R}M)$, one can see $ms\in\phi(X).$ By the same reason, as
	$r+s\notin(X:_{R}M)$, $m(r+s)\in\phi(X).$ Since $ms\in\phi(X),$ we obtain
	$mr\in\phi(X).$
	
	Now, let $m\in Y\cap X.$ Since $Y\nsubseteq X,$ there exists $m^{\ast}\in
	Y-X.$ By the above observations, $m^{\ast}r\in$ $\phi(X)$ and $(m+m^{\ast
	})r\in$ $\phi(X)$ (since $m+m^{\ast}\in Y-X$). This implies that $mr\in
	\phi(X).$
	
	Consequently, for every case we get $YI\subseteq\phi(X).$
\end{proof}

\begin{theorem}
	\label{the def}	For $X\in S(M)$, the items are equivalent:
	
	\begin{enumerate}
		\item $X$ is $\phi$\textit{-}prime.
		
		\item For $\forall$ right ideal $I$ in $R$ and $Y\in S(M),$
		\[
		YI\subseteq X\text{ and }YI\nsubseteq\phi(X)\text{ implies that }Y\subseteq
		X\text{ or }I\subseteq(X:_{R}M).
		\]

		\item For $\forall$ left ideal $I$ of $R$ and $Y\in S(M),$
		\[
		YI\subseteq X\text{ and }YI\nsubseteq\phi(X)\text{ implies that }Y\subseteq
		X\text{ or }I\subseteq(X:_{R}M).
		\]

		\item For $\forall a\in R$ and $Y\in S(M),$
		\[
		Y(RaR)\subseteq X\text{ and }Y(RaR)\nsubseteq\phi(X)\text{ implies that
		}Y\subseteq X\text{ or }a\in(X:_{R}M).
		\]

		\item For $\forall a\in R$ and $Y\in S(M),$
		\[
		Y(aR)\subseteq X\text{ and }Y(aR)\nsubseteq\phi(X)\text{ implies that
		}Y\subseteq X\text{ or }a\in(X:_{R}M).
		\]

		\item For $\forall a\in R$ and $Y\in S(M),$
		\[
		Y(Ra)\subseteq X\text{ and }Y(Ra)\nsubseteq\phi(X)\text{ implies that
		}Y\subseteq X\text{ or }a\in(X:_{R}M).
		\]
		
	\end{enumerate}
\end{theorem}

\begin{proof}
	$(1)\Rightarrow(2):$ Suppose that $X$ is $\phi$-prime. Choose a right
	ideal\ $I$ and $Y\in S(M)$ with $YI\subseteq X$, $YI\nsubseteq\phi(X).$ Let
	$<I>:=\{\sum r_{i}a_{i}s_{i}:r_{i},s_{i}\in R$ and $a_{i}\in I\}$ be the ideal
	generated by $I.$ Then as $I$ is a right ideal, one easily has that
	$Y<I>\subseteq YI\subseteq X.$ Moreover, $Y<I>\nsubseteq\phi(X).$ Indeed, if
	$Y<I>\subseteq\phi(X),$ then $YI\subseteq Y<I>\subseteq\phi(X),$ a
	contradiction. Thus, since $X$ is $\phi$-prime, $Y<I>\subseteq X$ and
	$Y<I>\nsubseteq\phi(X)$, we have $Y\subseteq X$ or $<I>\subseteq(X:_{R}M),$ so
	$I\subseteq(X:_{R}M).$
	
	$(2)\Rightarrow(3):$ Choose a left ideal\ $I$ and $Y\in S(M)$ with
	$YI\subseteq X$, $YI\nsubseteq\phi(X).$ Let consider again the ideal $<I>$ of
	$R.$ Then since $YI\subseteq X$ and $I$ is a left ideal, one can see that
	$Y<I>\subseteq X.$ Moreover, let us prove $Y<I>\nsubseteq\phi(X).$ Asumme that
	$Y<I>\subseteq\phi(X),$ then $YI\subseteq Y<I>\subseteq\phi(X),$ a
	contradiction. Thus, since $<I>$ is an ideal (so right ideal) by (2), we
	obtain $Y\subseteq X$ or $<I>\subseteq(X:_{R}M),$ so $I\subseteq(X:_{R}M).$
	
	$(3)\Rightarrow(4):$ Let $a\in R$ and $Y$ be a submodule of $M$ such that
	$Y(RaR)\subseteq X$ and $Y(RaR)\nsubseteq\phi(X).$ Since $Y=YR,$
	$Y(RaR)=YR(aR)=Y(Ra)\subseteq X$ and $Y(Ra)\nsubseteq\phi(X)$. Since $Ra$ is a
	left ideal, by (3), one can see $Y\subseteq X$ or $Ra\subseteq(X:_{R}M)$. Thus
	$Y\subseteq X$ or $a\in(X:_{R}M)$.
	
	$(4)\Rightarrow(5):$ Assume $a\in R$ and $Y\in S(M)$ with $Y(aR)\subseteq X$
	and $Y(aR)\nsubseteq\phi(X).$ Then we see $Y(aR)=YR(aR)\subseteq X$ and
	$YR(aR)\nsubseteq\phi(X)$. By (4), one obtains $Y\subseteq X$ or $a\in
	(X:_{R}M)$.
	
	$(5)\Rightarrow(6):$ Let $a\in R$ and $Y\in S(M)$ with $Y(Ra)\subseteq X$,
	$Y(Ra)\nsubseteq\phi(X).$ Thus $Ya\subseteq X$ and $Ya\nsubseteq\phi(X).$ Then
	we see $Y(aR)\subseteq X$ and $Y(aR)\nsubseteq\phi(X).$ Thus by (5),
	$Y\subseteq X$ or $a\in(X:_{R}M)$.
	
	$(6)\Rightarrow(1):$ Suppose that (6) satisfies. By the help of
	$(1)\Leftrightarrow(2)$ in Theorem \ref{the main}, let us prove that for all
	$m\in M-X,$ one has $(X:_{R}mR)=(X:_{R}M)\cup(\phi(X):_{R}mR).$ Let
	$a\in(X:_{R}mR).$ Then we see $mRa\subseteq X.$ If $mRa\subseteq\phi(X),$ one
	gets $a\in(\phi(X):_{R}mR).$ If $mRa\nsubseteq\phi(X),$ this implies that
	$(mR)(Ra)\nsubseteq\phi(X).$ Thus we have $mRa=(mR)(Ra)\subseteq X$ and
	$(mR)(Ra)\nsubseteq\phi(X).$ Then by (6), $mR\subseteq X$ or $a\in(X:_{R}M)$.
	The first option gives us a contradiction with $m\in M-X.$ Then $a\in
	(X:_{R}M).$ Thus $(X:_{R}mR)\subseteq(X:_{R}M)\cup(\phi(X):_{R}mR).$ Since the
	other containment always satisfies, we have $(X:_{R}mR)=(X:_{R}M)\cup
	(\phi(X):_{R}mR).$ Therefore, $X$ is a $\phi$\textit{-}prime submodule of $M.$
\end{proof}

\begin{theorem}
	\label{the ilk} If $X$ is a $\phi$\textit{-}prime submodule such that
	$X(X:_{R}M)\nsubseteq\phi(X),$ then $X$ is prime.
\end{theorem}

\begin{proof}
	Assume that $I$\ is an ideal of $R$ and $Y$ is a submodule of $M$ such that
	$YI\subseteq X.$ Then we have 2 cases:
	
	Case 1: $YI\nsubseteq\phi(X).$ As $X$ is $\phi$\textit{-}prime, we get
	$Y\subseteq X$ or $I\subseteq(X:_{R}M).$ So, it is done.
	
	Case 2: $YI\subseteq\phi(X).$ In this case, we may assume $XI\subseteq
	\phi(X)\cdot\cdot\cdot\cdot\cdot\cdot(1).$ Indeed, if $XI\nsubseteq\phi(X),$
	then there is an $m\in X$ such that $mI\nsubseteq\phi(X).$ Then we obtain
	$(Y+mR)I\subseteq X-\phi(X).$ As $X$ is $\phi$-prime, $Y+mR\subseteq X$ or
	$I\subseteq(X:_{R}M).$ So, $Y\subseteq X$ or $I\subseteq(X:_{R}M).$ Moreover,
	we may suppose $Y(X:_{R}M)$ $\subseteq\phi(X)\cdot\cdot\cdot\cdot\cdot
	\cdot(2).$ Indeed, if $Y(X:_{R}M)$ $\nsubseteq\phi(X),$ there exists an
	$a\in(X:_{R}M)$ with $Ya\nsubseteq\phi(X).$ Then we have $Y(I+RaR)\subseteq X$
	and $Y(I+RaR)\nsubseteq\phi(X).$ Since $X$ is $\phi$-prime, $Y\subseteq X$ or
	$I+RaR\subseteq(X:_{R}M).$ Therefore, $Y\subseteq X$ or $I\subseteq(X:_{R}M).$
	
	As $X(X:_{R}M)\nsubseteq\phi(X),$ one can see that there are $b\in(X:_{R}M)$ and
	$x\in X$ such that $xb\notin\phi(X).$ Then by (1) and (2), we obtain
	$(Y+xR)(I+RbR)\subseteq X$ and $(Y+xR)(I+RbR)\nsubseteq\phi(X).$ By the help
	of the hypothesis, $Y+xR\subseteq X$ or $I+RbR\subseteq(X:_{R}M).$ Then one
	obtains $Y\subseteq X$ or $I\subseteq(X:_{R}M).$
\end{proof}

\begin{corollary}
	If $X$ is a weakly prime submodule with $X(X:_{R}M)\neq0_{M},$ then $X$ is prime.
\end{corollary}

\begin{proof}
	In Theorem \ref{the ilk}, set $\phi=\phi_{0}.$
\end{proof}

\begin{corollary}
	If $X$ is a $\phi$\textit{-}prime submodule such that $\phi(X)\subseteq
	X(X:_{R}M)^{2},$ then $X$ is $\phi_{\omega}$-prime.
\end{corollary}

\begin{proof}
	Assume that $YI\subseteq X$ and $YI\nsubseteq\cap_{i=1}^{\infty}X(X:_{R}%
	M)^{i},$ for some $Y\in S(M)$ and ideal $I$ of $R$. If $X$ is prime, we are
	done. So, suppose $X$ is not prime. Then Theorem \ref{the ilk} implies
	$X(X:_{R}M)\subseteq\phi(X)\subseteq X(X:_{R}M)^{2}$ $\subseteq X(X:_{R}M),$
	i.e., $X(X:_{R}M)=\phi(X)=X(X:_{R}M)^{2}.$ Thus, we obtain $\phi(X)=\cap
	_{i=1}^{\infty}X(X:_{R}M)^{i},$ for every $i\geq1.$ As $X$ is $\phi$
	\textit{-}prime, $Y\subseteq X$ or $I\subseteq(X:_{R}M)$. Consequently, we
	obtain $X$ is $\phi_{\omega}$-prime.
\end{proof}

\bigskip

Note that a submodule $X$ of $M$ is called \textit{radical} if $\sqrt
{(X:_{R}M)}=(X:_{R}M).$

\begin{corollary}
	Let $X$ be a $\phi$\textit{-}prime submodule of $M$. Then
	\begin{enumerate}
		\item Either $(X:_{R}M)\subseteq\sqrt{(\phi(X):_{R}M)}$ or $\sqrt
		{(\phi(X):_{R}M)}\subseteq(X:_{R}M).$
		
		\item If $(X:_{R}M)\subsetneq\sqrt{(\phi(X):_{R}M)},$ $X$ is not prime$.$
		
		\item If $\sqrt{(\phi(X):_{R}M)}\subsetneq(X:_{R}M),$ $X$ is prime$.$
		
		\item If $\phi(X)$ is a radical submodule, then either $(X:_{R}M)=(\phi
		(X):_{R}M)$ or $X$ is prime.
	\end{enumerate}
\end{corollary}

\begin{proof}
	Suppose $X$ is $\phi$\textit{-}prime.
	\begin{enumerate}
		\item Assume that $X$ is prime. Then $(X:_{R}M)$ is a prime ideal of $R$, see
		\cite{D}. As $\phi(X)\subseteq X,$ we see $(\phi(X):_{R}M)\subseteq(X:_{R}M),$
		so $\sqrt{(\phi(X):_{R}M)}\subseteq\sqrt{(X:_{R}M)}=(X:_{R}M).$ Now assume
		that $X$ is not prime. By Theorem \ref{the ilk}, one see $X(X:_{R}%
		M)\subseteq\phi(X).$ This implies that $\sqrt{(X:_{R}M)^{2}}\subseteq
		\sqrt{(X(X:_{R}M):_{R}M)}\subseteq\sqrt{(\phi(X):_{R}M)}.$ Hence
		$(X:_{R}M)\subseteq\sqrt{(X:_{R}M)}=\sqrt{(X:_{R}M)^{2}}\subseteq\sqrt
		{(\phi(X):_{R}M)}.$
		
		\item Suppose $(X:_{R}M)\subsetneq\sqrt{(\phi(X):_{R}M)}.$ If $X$ is prime,
		$\sqrt{(\phi(X):_{R}M)}\subseteq\sqrt{(X:_{R}M)}=(X:_{R}M),$ i.e., a
		contradiction. So, $X$ is not prime.
		
		\item Let $\sqrt{(\phi(X):_{R}M)}\subsetneq(X:_{R}M).$ If $X$ is not prime, by
		the help of Theorem \ref{the ilk}, we get $X(X:_{R}M)\subseteq\phi(X).$ Then
		one see $\sqrt{(X:_{R}M)^{2}}\subseteq\sqrt{(X(X:_{R}M):_{R}M)}\subseteq
		\sqrt{(\phi(X):_{R}M)}.$ Hence, since $\sqrt{(X:_{R}M)^{2}}=	\sqrt{(X:_{R}M)},$ $(X:_{R}M)\subseteq\sqrt{(\phi(X):_{R}M)},$
		i.e., a contradiction.
		
		\item  Let $\phi(X)$ be a radical submodule. Suppose that $X$ is not prime. By
		the argument in the proof of (1), $(X:_{R}M)\subseteq\sqrt{(\phi(X):_{R}M)}.$
		Then since $\phi(X)$ is a radical submodule, we see that  $(X:_{R}%
		M)\subseteq\sqrt{(\phi(X):_{R}M)}=(\phi(X):_{R}M).$ As the other containment
		is always hold, $(X:_{R}M)=(\phi(X):_{R}M).$
	\end{enumerate}
\end{proof}

\begin{remark}
	Assume that $X\in S(M).$
	
	\begin{enumerate}
		\item If $X$ is $\phi$-prime but not prime such that $\phi(X)\subseteq
		X(X:_{R}M),$ then $\phi(X)=X(X:_{R}M).$ In particular, if $X$ is\ not prime
		and $X$ is weakly prime, then $X(X:_{R}M)=0_{M}.$
		
		\item If $X$ is $\phi$-prime but not prime such that $\phi(X)\subseteq
		X(X:_{R}M)^{2},$ then $\phi(X)=X(X:_{R}M)^{2}.$ In particular, if $X$ is\ not
		prime and $X$ is $\phi_{2}$-prime, then $X(X:_{R}M)=X(X:_{R}M)^{2}.$
	\end{enumerate}
\end{remark}

\bigskip

Now, for $Y\in S(M),$ let us define $\phi_{Y}:S(M/Y)\rightarrow S(M/Y)\cup
\{\emptyset\}$ by $\phi_{Y}(X/Y)=(\phi(X)+Y)/Y,$ for every $X\in S(M)$ with
$Y\subseteq X$ (and $\phi_{Y}(X/Y)=\emptyset$ if $\phi(X)=\emptyset$).

\bigskip

\begin{theorem}
	\label{the bolum} Let $X,Y\in S(M)$ be proper with $Y\subseteq X.$ Then we have
	
	\begin{enumerate}
		\item If $X$ is a $\phi$-prime submodule of $M,$ then $X/Y$ is a $\phi_{Y}
		$-prime submodule of $M/Y$.
		
		\item If $Y\subseteq\phi(X)$ and $X/Y$ is a $\phi_{Y}$-prime submodule of
		$M/Y,$ then $X$ is a $\phi$-prime submodule of $M$.
		
		\item If $\phi(X)\subseteq Y$ and $X$ is $\phi$-prime, then $X/Y$ is weakly prime.
		
		\item If $\phi(Y)\subseteq\phi(X),$ $Y$ is $\phi$-prime and $X/Y$ is weakly
		prime, then $X$ is $\phi$-prime.
	\end{enumerate}
\end{theorem}

\begin{proof}
	Let $X,Y\in S(M)$ be proper with $Y\subseteq X.$
	
	(1) : Assume $I\in S(R)$ and $Z/Y$ is a submodule of $M/Y$ with
	$(Z/Y)I\subseteq X/Y$ and $(Z/Y)I\nsubseteq$ $\phi_{Y}(X/Y).$ Then clearly,
	$(Z/Y)I=ZI+Y/Y$ and $ZI\subseteq ZI+Y\subseteq X.$ Moreover $ZI\nsubseteq
	\phi(X).$ Indeed, if $ZI\subseteq\phi(X),$ then one can see $(ZI+Y)/Y\subseteq
	(\phi(X)+Y)/Y=\phi_{Y}(X/Y),$ so $(Z/Y)I\subseteq\phi_{Y}(X/Y)$, i.e., a
	contradiction. Since $X$ is $\phi$-prime, we see $I\subseteq(X:_{R}M)$ or
	$Z\subseteq X.$ Then one obtains $I\subseteq(X:_{R}M)=(X/Y:_{R}M/Y)$ or
	$Z/Y\subseteq X/Y.$
	
	(2) : Suppose that $I$ is an ideal of $R$ and $Z$ is a submodule of $M$ such
	that $ZI\subseteq X$ and $ZI\nsubseteq\phi(X).$ Then $ZI+Y/Y=(Z/Y)I\subseteq
	X/Y.$ Moreover, $(Z/Y)I\nsubseteq\phi_{Y}(X/Y).$ Indeed, if $(Z/Y)I\subseteq
	\phi_{Y}(X/Y)=(\phi(X)+Y)/Y,$ as $Y\subseteq\phi(X)$ we have $ZI+Y/Y\subseteq
	\phi(X)/Y,$ i.e., $ZI\subseteq\phi(X),$ a contradiction. Since $X/Y$ is a
	$\phi_{Y}$-prime submodule of $M/Y,$ one can see $I\subseteq(X/Y:_{R}M/Y)$ or
	$Z/Y\subseteq X/Y.$ This implies that $I\subseteq(X:_{R}M)$ or $Z\subseteq X.$
	
	(3) : Assume that $I\in S(R)$ and $Z/Y$ is a submodule of $M/Y$ with
	$0_{M/Y}\neq(Z/Y)I\subseteq X/Y.$ Clearly, we have $Y\subset ZI\subseteq X.$
	Then since $\phi(X)\subseteq Y,$ we see $ZI\nsubseteq\phi(X).$ As $X$ is
	$\phi$-prime, $I\subseteq(X:_{R}M)$ or $Z\subseteq X.$ This implies
	$I\subseteq(X/Y:_{R}M/Y)$ or $Z/Y\subseteq X/Y.$
	
	(4) : Suppose that $\phi(Y)\subseteq\phi(X)$, $Y$ is $\phi$-prime and $X/Y$ is
	weakly prime. Choose $Z\in S(M)$ and an ideal $I$ of $R$ which $ZI\subseteq
	X$, $ZI\nsubseteq\phi(X).$ Then since $\phi(Y)\subseteq\phi(X)$ and
	$ZI\nsubseteq\phi(X),$ we have $ZI\nsubseteq\phi(Y).$ Then one can see 2 cases :
	
	Case 1 : $ZI\subseteq Y.$ As $Y$ is $\phi$-prime, $I\subseteq(Y:_{R}M)$ or
	$Z\subseteq Y.$ Since $Y\subseteq X,$ we have $I\subseteq(X:_{R}M)$ or
	$Z\subseteq X,$ so it is done.
	
	Case 2 : $ZI\nsubseteq Y.$ Then $0_{M/Y}\neq ZI+Y/Y=(Z/Y)I\subseteq X/Y.$
	Since $X/Y$ is weakly prime, $I\subseteq(X/Y:_{R}M/Y)$ or $Z/Y\subseteq X/Y.$
	Thus, we obtain $I\subseteq(X:_{R}M)$ or $Z\subseteq X.$
\end{proof}

\begin{corollary}
	For a proper $X\in S(M)$, $X$ is $\phi$-prime in $M$ $\Longleftrightarrow$
	$X/\phi(X)$ is weakly prime in $M/\phi(X).$
\end{corollary}

\begin{proof}
	$\Longrightarrow:$ By (3) of Theorem \ref{the bolum}.
	
	$\Longleftarrow:$ By (2) of Theorem \ref{the bolum}.
\end{proof}

\bigskip

Note that we say $M$ is a torsion-free module if $(0_{M}:_{R}m)=0_{R},$ for
all $0_{M}\neq m\in M.$

\begin{theorem}
	\label{the Rm}Let $M$ be torsion-free and $0_{M}\neq m\in M$. Then $mR$ is
	prime $\Longleftrightarrow$ $mR$ is almost prime.
\end{theorem}

\begin{proof}
	$\Longrightarrow:$ Obvious.
	
	$\Longleftarrow:$ Assume that $mR$ is not prime. Then there are $a\in R$,
	$x\in M$ with $a\notin(mR:_{R}M)$, $x\notin mR,$ also $xRa\subseteq mR.$ Then
	we have $(xR)(RaR)\subseteq mR$ and the following 2 cases:
	
	Case 1 : $(xR)(RaR)\nsubseteq mR(mR:_{R}M)=\phi_{2}(mR).$ Since $a\notin
	(mR:_{R}M)$, $x\notin mR$, one gets $(RaR)\nsubseteq(mR:_{R}M)$ and
	$(xR)\nsubseteq mR.$ Thus we obtain that $mR$ is not almost prime.
	
	Case 2 : $(xR)(RaR)\subseteq mR(mR:_{R}M)=\phi_{2}(mR).$ Then we have $xa\in$
	$mR(mR:_{R}M).$ Moreover, as $xRa\subseteq mR,$ we have $(x+m)a\in mR$ and $x+m\notin mR.$ Then $(xR+mR)(RaR)\subseteq mR.$ If
	$(xR+mR)(RaR)\nsubseteq mR(mR:_{R}M),$ as $a\notin(mR:_{R}M)$ and $x+m\notin
	mR,$ one can see $mR$ is not almost prime. If $(xR+mR)(RaR)\subseteq
	mR(mR:_{R}M),$ then $(x+m)a\in mR(mR:_{R}M).$ Then, by the assumption in Case
	2, we have $xa\in mR(mR:_{R}M),$ so, $ma\in mR(mR:_{R}M).$ Hence there exist
	an element $b\in(mR:_{R}M)$ and $r\in R$ such that $ma=(mr)b.$ This implies
	that $a-rb\in(0_{M}:_{R}m)=0_{R},$ i.e., $a=rb\in(mR:_{R}M).$ So, we obtain a
	contradiction with $a\notin(mR:_{R}M).$ Consequently, in every case $mR$ is
	not almost prime.
\end{proof}

\begin{theorem}
	\label{the aM}Let $0_{R}\neq a\in R$ such that $(0_{M}:_{M}a)\subseteq Ma$ and
	$a(Ma:_{R}M)=(Ma:_{R}M)a.$ Thus $Ma$ is prime $\Longleftrightarrow$ $Ma$ is
	almost prime.
\end{theorem}

\begin{proof}
	$\Longrightarrow:$ It is obvious.
	
	$\Longleftarrow:$ Suppose that $Ma$ is almost prime. Let $b\in R$, $m\in
	M$ with $mRb\subseteq Ma.$ We prove that $m\in Ma$ or $b\in(Ma:_{R}M).$ Then
	one can see clearly, $(mR)(RbR)\subseteq Ma.$ Now, we get 2 cases:
	
	Case 1 : $(mR)(RbR)\nsubseteq Ma(Ma:_{R}M)=\phi_{2}(Ma).$ Since $Ma$ is
	almost prime, we have $mR\subseteq Ma$ or $RbR\subseteq(Ma:_{R}M).$ So,
	$m\in Ma$ or $b\in(Ma:_{R}M).$
	
	Case 2 : $(mR)(RbR)\subseteq Ma(Ma:_{R}M)=\phi_{2}(Ma).$ As $mb\in Ma,$ one
	gets $m(b+a)\in Ma.$ Then $(mR)(RbR+RaR)\subseteq Ma.$ If
	$(mR)(RbR+RaR)\nsubseteq Ma(Ma:_{R}M),$ as $Ma$ is almost prime,
	$mR\subseteq Ma$ or $RbR+RaR\subseteq(Ma:_{R}M).$ Thus, one can see
	$mR\subseteq Ma$ or $RbR\subseteq(Ma:_{R}M).$ Therefore, it is done. If
	$(mR)(RbR+RaR)\subseteq Ma(Ma:_{R}M),$ then $(mR)(RaR)\subseteq Ma(Ma:_{R}%
	M)=M(Ma:_{R}M)a.$ Thus $ma\in M(Ma:_{R}M)a.$ Then, one has $n\in M(Ma:_{R}M)$
	with $ma=na.$ Hence $m-n\in(0_{M}:_{M}a)\subseteq Ma.$ This implies $m\in
	M(Ma:_{R}M)+(0_{M}:_{M}a)\subseteq Ma.$
\end{proof}

\begin{corollary}
	Let $M$ be torsion-free and $a\in R$ such that $a(Ma:_{R}M)=(Ma:_{R}M)a.$ Thus
	$Ma$ is prime $\Longleftrightarrow$ $Ma$ is almost prime.
\end{corollary}

\begin{proof}
	By Theorem \ref{the aM}, it is clear.
\end{proof}

\begin{theorem}
	\label{the cha}Let $X$ be a proper submodule of $M$. Then the followings are equivalent:
	
	\begin{enumerate}
		\item $X$ is a $\phi$-prime submodule of $M.$
		
		\item For all ideal $I$ of $R$ with $I\nsubseteq(X:_{R}M),$ then
		
		$(X:_{M}I)=X\cup(\phi(X):_{M}I).$
		
		\item For all ideal $I$ of $R$ with $I\nsubseteq(X:_{R}M),$ then
		
		$(X:_{M}I)=X$ or $(X:_{M}I)=(\phi(X):_{M}I).$
	\end{enumerate}
\end{theorem}

\begin{proof}
	Choose $X\in S(M).$
	
	$(1)\Longrightarrow(2)$ : Assume $X$ is $\phi$-prime$.$ Choose an ideal $I$
	which $I\nsubseteq(X:_{R}M).$ Then one can see $X\subseteq(X:_{M}I)$ and
	$(\phi(X):_{M}I)\subseteq(X:_{M}I),$ so $X\cup(\phi(X):_{M}I)\subseteq
	(X:_{M}I).$ For the other containment, since $(X:_{M}I)I\subseteq X,$ and
	one gets 2 cases:
	
	Case 1: $(X:_{M}I)I\nsubseteq\phi(X).$ Then since $(X:_{M}I)I\subseteq X$ and
	$X$ is $\phi$-prime, $I\subseteq(X:_{R}M)$ or $(X:_{M}I)\subseteq X.$ As the
	first option gives us a contradiction, it must be $(X:_{M}I)\subseteq X.$
	
	Case 2: $(X:_{M}I)I\subseteq\phi(X).$ Then we obtain $(X:_{M}I)\subseteq
	(\phi(X):_{M}I),$ so it is done.
	
	$(2)\Longrightarrow(3)$ : If a submodule is a union of two submodules, it
	equals to one of them.
	
	$(3)\Longrightarrow(1)$ : Choose an ideal $I$ in $R$, $Y\in S(M)$ with
	$YI\subseteq X$, $YI\nsubseteq\phi(X).$ If $I\subseteq(X:_{R}M),$ it is done.
	Suppose $I\nsubseteq(X:_{R}M).$ Then by (3), one can see $(X:_{M}I)=X$ or
	$(X:_{M}I)=(\phi(X):_{M}I).$ If $(X:_{M}I)=X,$ since $YI\subseteq X,$ we
	have $Y\subseteq(X:_{M}I)=X.$ So, we are done. If $(X:_{M}I)=(\phi(X):_{M}I),$
	as $YI\nsubseteq\phi(X),$ we have $Y\nsubseteq(\phi(X):_{M}I)=(X:_{M}I),$ a
	contradiction with $YI\subseteq X.$
\end{proof}

\begin{proposition}
	\label{pro IM}Let $X$ be a proper submodule of $M$ and $I$ be an ideal of $R$
	such that $MI\neq XI$ and $XI\neq X.$ Then $Y=XI$ is a $\phi$-prime submodule
	of $M$ if and only if $Y=\phi(Y).$
\end{proposition}

\begin{proof}
	$\Longleftarrow:$ Let $Y=\phi(Y).$ Then obviously $Y$ is $\phi$-prime.
	
	$\Longrightarrow:$ Suppose that $Y=XI$ is a $\phi$-prime submodule. Let us
	consider Theorem \ref{the cha}. Now, we have 2 cases:
	
	Case 1 : $I$ $\nsubseteq(Y:_{R}M).$ By Theorem \ref{the cha}, one obtains
	$(Y:_{M}I)=Y$ or $(Y:_{M}I)=(\phi(Y):_{M}I).$ If $(Y:_{M}I)=Y,$ we have
	$X\subseteq(Y:_{M}I)=(XI:_{M}I)=Y=XI$, i.e., $X=XI$, a contradiction. If
	$(Y:_{M}I)=(\phi(Y):_{M}I),$ as $X\subseteq(Y:_{M}I),$ we see $Y=XI\subseteq
	(Y:_{M}I)I=(\phi(Y):_{M}I)I\subseteq\phi(Y),$ so $Y\subseteq\phi(Y).$ Then one
	obtains $\phi(Y)=Y.$ So it is done.
	
	Case 2 : $I$ $\subseteq(Y:_{R}M).$ Then $MI\subseteq Y=XI,$ so $MI=XI,$ a contradiction.
\end{proof}

\begin{corollary}
	Let $X$ be a proper submodule of $M$ and $I$ be an ideal of $R$ such that
	$MI^{n}\neq MI^{n-1}$ for some $n>1.$ Then $Y=MI^{n}$ is a $\phi$-prime
	submodule of $M$ if and only if $Y=\phi(Y).$
\end{corollary}

\begin{proof}
	Let consider $X=MI^{n-1}.$ Then $XI=MI^{n}\subsetneq MI^{n-1}\subseteq MI,$
	i.e., $XI\neq MI.$ Moreover,\ $Y=XI=MI^{n}\neq MI^{n-1}=X,$ i.e., $XI\neq X.$
	Thus, by Proposition \ref{pro IM}, it is done.
\end{proof}

\begin{proposition}
	Let $I$ be a maximal ideal in $R.$ Then $MI=M$ or $MI$ is $\phi$-prime in $M.$
\end{proposition}

\begin{proof}
	Let $MI\neq M.$ By the proof of Proposition 2.12 in \cite{Beiranvand}, one can
	see that $MI$ is a prime submodule of $M.$ Thus, $MI$ is $\phi$-prime.
\end{proof}

\begin{theorem}
	\label{the ideal}Let $X$ be a proper submodule of $M.$ Suppose that
	$\psi:S(R)\rightarrow S(R)\cup\{\emptyset\}$ be a function. If $X$ is $\phi
	$-prime, then $(X:_{R}Y)$ is a $\psi$-prime ideal of $R,$ for all $Y\in S(M)$
	with $Y\nsubseteq X$ and $(\phi(X):_{R}Y)\subseteq\psi((X:_{R}Y))$.
\end{theorem}

\begin{proof}
	Suppose that $X$ is a $\phi$-prime submodule of $M$ and $Y$ is a submodule of
	$M$ such that $Y\nsubseteq X$ and $(\phi(X):_{R}Y)\subseteq\psi((X:_{R}Y))$.
	Let $IJ\subseteq\left(  X:_{R}Y\right)  $ and $IJ\nsubseteq\psi((X:_{R}Y))$
	for two ideals $I,J$ of $R.$ Then $(YI)J\subseteq X$ and $(YI)J\nsubseteq
	\phi(X),$ since $(\phi(X):_{R}Y)\subseteq\psi((X:_{R}Y)).$ By our hypothesis,
	$J\subseteq(X:_{R}M)$ or $YI\subseteq{X}$. If $YI\subseteq{X,}$ i.e.,
	$I\subseteq(X:_{R}Y)$, it is done$.$ If $J\subseteq(X:_{R}M)$, since
	$(X:_{R}M)\subseteq(X:_{R}Y),$ we see $J\subseteq(X:_{R}Y)$. Consequently,
	$\left(  X:_{R}Y\right)  $ is a $\psi$-prime ideal of $R$.
\end{proof}

\begin{corollary}
	\label{cor ideal}Let $X$ be a proper submodule of $M.$ Suppose that
	$\psi:S(R)\rightarrow S(R)\cup\{\emptyset\}$ be a function with $(\phi
	(X):_{R}M)\subseteq\psi((X:_{R}M)).$ If $X$ is a $\phi$-prime submodule of
	$M,$ then $(X:_{R}M)$ is a $\psi$-prime ideal of $R.$
\end{corollary}

\begin{proof}
	Set $Y=M$ in Theorem \ref{the ideal}.
\end{proof}

\section{$\phi-$Prime submodules in multiplication modules}

\label{Sec:3}

Note that, an $R$-module $M$ is called a \textit{multiplication module} if
there is an ideal $I$ of $R$ such that $X=MI,$ for all $X\in S(M),$ see
\cite{tug}. Also, in a multiplication module, one can see $X=M(X:_{R}M)$,\ for
all $X\in S(M),$ see \cite{tug}.

Let $X$ and $Y$ be two submodules of a multiplication $R$-module $M$ with
$X=M(X:_{R}M)$ and $Y=M(Y:_{R}M)$. The product of $X$ and $Y$ is denoted by
$XY$ and it is defined by $XY=M(X:_{R}M)(Y:_{R}M)$. It is clear that the
product is well-defined.

\begin{proposition}
	\label{pro multip}Let $M$ be multiplication and $X\in S(M).$ Then if $X$ is
	$\phi$-prime, then for $Y_{1}$, $Y_{2}\in S(M),$ $Y_{1}Y_{2}\subseteq X$ and
	$Y_{1}Y_{2}\nsubseteq\phi(X)$ implies that $Y_{1}\subseteq X$ or
	$Y_{2}\subseteq X.$
\end{proposition}

\begin{proof}
	Let $Y_{1}$, $Y_{2}$ be any submodule in $M$ with $Y_{1}Y_{2}\subseteq X$ and
	$Y_{1}Y_{2}\nsubseteq\phi(X).$ As $M$ is multiplication, we know that
	$Y_{1}=M(Y_{1}:_{R}M)$ and $Y_{2}=M(Y_{2}:_{R}M)$. Then $Y_{1}Y_{2}%
	=M(Y_{1}:_{R}M)(Y_{2}:_{R}M)\subseteq X$ and $Y_{1}Y_{2}\nsubseteq\phi(X).$
	Since $X$ is $\phi$-prime, one can see $M(Y_{1}:_{R}M)\subseteq X$ or
	$(Y_{2}:_{R}M)\subseteq(X:_{R}M).$ This implies that $Y_{1}\subseteq X$ or
	$Y_{2}=M(Y_{2}:_{R}M)\subseteq M(X:_{R}M)=X.$
\end{proof}
\bigskip Note that we say $M$ is a cancellation module if $MI=MJ$ implies that $I=J$ for two ideals $I,J$ of $R.$ For the definition of a cancellation
module over commutative ring, see \cite{cancel}.

\begin{corollary}
	\label{cor multip}	Let $M$ be multiplication and cancellation. For $X\in S(M),$ the statements
	are equivalent:
	\begin{enumerate}
		\item $X$ is $\phi$-prime.
		
		\item For $Y_{1}$, $Y_{2}\in S(M),$ if $Y_{1}Y_{2}\subseteq X$ and $Y_{1}
		Y_{2}\nsubseteq\phi(X),$ then $Y_{1}\subseteq X$ or $Y_{2}\subseteq X.$
	\end{enumerate}
	
\end{corollary}

\begin{proof}
	$(1)\Longrightarrow(2)$ : By Proposition \ref{pro multip}.
	
	$(2)\Longrightarrow(1)$ : Choose an ideal $I\in S(R)$, $Y\in S(M)$ with
	$YI\subseteq X$ and $YI\nsubseteq\phi(X).$ Since $M$ is multiplication,
	$Y=M(Y:_{R}M).$ Then we have $M(Y:_{R}M)I=YI\subseteq X$ and $YI\nsubseteq
	\phi(X).$ Also, as $M$ is multiplication, $MI=M(MI:_{R}M).$ Then this implies
	that $I=(MI:_{R}M),$ since $M$ is cancellation. Hence $Y(MI)=M(Y:_{R}%
	M)(MI:_{R}M)=M(Y:_{R}M)I=YI.$ So, we have $Y(MI)\subseteq X$ and
	$Y(MI)\nsubseteq\phi(X).$ Then by (2), one see $Y\subseteq X$ or $MI\subseteq
	X.$ This means that $Y\subseteq X$ or $I\subseteq(X:_{R}M).$
\end{proof}

\begin{theorem}
	\label{the ideal iff}Let $M$ be a multiplication $R$-module and $X$ be a
	proper submodule of $M.$ Suppose that $\psi:S(R)\rightarrow S(R)\cup
	\{\emptyset\}$ be a function with $(\phi(X):_{R}M)=\psi((X:_{R}M))$. Then the
	followings are equivalent:
	
	\begin{enumerate}
		\item $X$ is $\phi$-prime in $M.$
		
		\item $(X:_{R}M)$ is a $\psi$-prime ideal in $R$.
	\end{enumerate}
\end{theorem}

\begin{proof}
	$(1)\Longrightarrow(2):$ By Corollary \ref{cor ideal}.
	
	$(2)\Longrightarrow(1)$ : Assume that $(X:_{R}M)$ is $\psi$-prime. Choose an
	ideal $I$ of $R$ and a submodule $Y$ of $M$ with $YI\subseteq X$ and
	$YI\nsubseteq\phi(X)$. As $M$ is multiplication, $Y=M(Y:_{R}M).$ Hence
	$M(Y:_{R}M)I\subseteq X$ and $M(Y:_{R}M)I\nsubseteq\phi(X).$ Then one gets
	$(Y:_{R}M)I\subseteq{(X:_{R}M)}$ and $(Y:_{R}M)I\nsubseteq(\phi(X):_{R}M).$
	Since $(\phi(X):_{R}M)=\psi((X:_{R}M))$, $(Y:_{R}M)I\nsubseteq\psi
	((X:_{R}M)).$ By our hypothesis, $I\subseteq(X:_{R}M)$ or $(Y:_{R}%
	M)\subseteq{(X:_{R}M).}$ If $I\subseteq(X:_{R}M),$ it is done. If $(Y:_{R}%
	M)\subseteq{(X:_{R}M),}$ as $M$ is multiplication, one can see $Y=M(Y:_{R}M)\subseteq M{(X:_{R}M)=X.}$ Therefore, $X$ is $\phi$-prime.
\end{proof}

\bigskip

Recall that if there exists an element $s\in R$ with $r=rsr,$ for all $r\in
R$, $R$ is called \textit{von-Neumann regular, }see\textit{ \cite{tug}. }Also,
\textit{the center of a ring} $R$ is denoted by $Center(R).$

\begin{lemma}
	\cite{Beiranvand}\label{lemma 1} Assume that $M$ is multiplication, $R$ is a
	von-Neumann regular ring and $J\subseteq Center(R)$ is an ideal in $R.$ Then
	$X\cap MJ=(X:_{M}J)J,$ for any submodule $X$ of $M$.
\end{lemma}

\begin{lemma}
	\cite{Beiranvand} \label{lemma 2}Assume that $M$ is multiplication, $R$ is a
	von-Neumann regular ring and $J\subseteq Center(R)$ is an ideal in $R.$ If for all $Y,Z\in S(M),$
	$YJ\subseteq ZJ$ implies that $Y\subseteq Z,$ then
	$(XI:_{M}J)=(X:_{M}J)I$ for $X\in S(MJ)$ and any ideal $I$ of $R$.
\end{lemma}

\begin{theorem}
	Let $M$ be a multiplication $R$-module and $R$ be a von-Neumann regular ring.
	Let $I\subseteq Center(R)$ be an ideal of $R$ such that $YI\subseteq ZI$
	implies that $Y\subseteq Z$ for all $Y,Z\in S(M).$ Let $\phi
	((X:_{M}I))=(\phi(X):_{M}I).$ Then $X\in S(MI)$ is $\phi$-prime
	$\Longleftrightarrow$ $(X:_{M}I)\in S(M)$ is $\phi$-prime$.$
\end{theorem}

\begin{proof}
	$\Longrightarrow:$ Assume that $X\in S(MI)$ is $\phi$-prime. Choose an ideal
	$J$ of $R$, $Y\in S(M)$ with $YJ\subseteq(X:_{M}I)$ and $YJ\nsubseteq
	\phi((X:_{M}I)).$ Then clearly $YJI\subseteq X.$ We show that $YJI\nsubseteq
	\phi(X).$ If $YJI\subseteq\phi(X),$ then\ $YJ\subseteq(\phi(X):_{M}%
	I)=\phi((X:_{M}I)),$ a contradiction. By $I\subseteq Center(R),$ one can see
	$YJI=YIJ.$ Hence, $YIJ\subseteq X$ and $YIJ\nsubseteq\phi$$(X)$ implies
	$YI\subseteq X$ or $J\subseteq(X:_{R}MI),$ since $X$ is $\phi$-prime submodule
	of $MI.$ Moreover, as $I\subseteq Center(R),$ we see $(X:_{R}MI)=((X:_{M}%
	I):_{R}M).$ So, $YI\subseteq X$ or $J\subseteq(X:_{R}MI)$ implies
	$Y\subseteq(X:_{M}I)$ or $J\subseteq((X:_{M}I):_{R}M).$
	
	$\Longleftarrow:$ Let $(X:_{M}I)$ be $\phi$-prime in $M$ for $X\in S(MI).$
	Choose an ideal $J$ of $R$, $Y\in S(MI)$ with $YJ\subseteq X$, $YJ\nsubseteq
	\phi(X).$ Then we see that $(Y:_{M}I)J=(YJ:_{M}I)$ $\subseteq(X:_{M}I)$ by
	Lemma \ref{lemma 2}. Now, let us prove $(Y:_{M}I)J\nsubseteq$ $\phi
	((X:_{M}I)).$ Indeed, if $(Y:_{M}I)J\subseteq$ $\phi((X:_{M}I))=(\phi
	(X):_{M}I),$ then $(Y:_{M}I)JI=(Y:_{M}I)IJ\subseteq(\phi(X):_{M}I)I,$ as
	$I\subseteq Center(R).$ By Lemma \ref{lemma 1}, we get $YJ=(Y\cap
	MI)J=(Y:_{M}I)IJ\subseteq(\phi(X):_{M}I)I=\phi(X)\cap MI$ $=\phi(X),$ a
	contradiction. Hence, as $(X:_{M}I)$ is $\phi$-prime, one can see
	$(Y:_{M}I)\subseteq(X:_{M}I)$ or $J\subseteq((X:_{M}I):_{R}M).$ The first
	option gives us $Y=Y\cap MI=(Y:_{M}I)I\subseteq(X:_{M}I)I=X$$\cap MI=X,$ by
	Lemma \ref{lemma 1}. The second option means that $J\subseteq((X:_{M}%
	I):_{R}M)=(X:_{R}MI),$ as $I\subseteq Center(R).$ Thus we are done.
\end{proof}

\section{The radical of a submodule}

\bigskip

In the following definition, we shall introduce the concept of $\phi$-$m$-system.

\begin{definition}
	\label{def system} $\emptyset\neq S\subseteq M$ is called a $\phi$-$m$-system
	if $(Y_{1}+Y_{2})\cap S\neq\emptyset,$ $(Y_{1}+MI)\cap S\neq\emptyset$ and
	$Y_{2}I\nsubseteq\phi(<S^{c}>),$ then $(Y_{1}+Y_{2}I)\cap S\neq\emptyset$ for
	$\forall Y_{1},Y_{2}\in S(M)$ and any ideal $I$ of $R$, where $S^{c}=M-S.$
\end{definition}

\begin{proposition}
	\label{system} For $X\in S(M)$, $X$ is $\phi$-prime $\Longleftrightarrow$
	$S=M-X$ is a $\phi$-$m$-system.
\end{proposition}

\begin{proof}
	$\Longrightarrow:$ Suppose that $X$ is $\phi$-prime. Choose an ideal $I$ of
	$R$ and two submodules $Y_{1}$, $Y_{2}$ of $M$ with $(Y_{1}+Y_{2})\cap
	S\neq\emptyset,$ $(Y_{1}+MI)\cap S\neq\emptyset$ and $Y_{2}I\nsubseteq
	\phi(<S^{c}>),$ where $S^{c}=X.$ We show that $(Y_{1}+Y_{2}I)\cap
	S\neq\emptyset.$ If $(Y_{1}+Y_{2}I)\cap S=\emptyset,$ then $(Y_{1}%
	+Y_{2}I)\subseteq X,$ since $S=M-X.$ Then one can see $Y_{2}I\subseteq X$ and
	$Y_{1}\subseteq X.$ Also, by our hypothesis, $Y_{2}I\nsubseteq\phi
	(<S^{c}>)=\phi(X).$ Then as $X$ is $\phi$-prime, we get $Y_{2}\subseteq X$ or
	$I\subseteq(X:_{R}M).$ If $Y_{2}\subseteq X$, we see $Y_{1}+Y_{2}\subseteq X$,
	i.e., $(Y_{1}+Y_{2})\cap S=\emptyset,$ a contradiction. If $I\subseteq
	(X:_{R}M),$ then $MI\subseteq X,$ so we get $Y_{1}+MI\subseteq X$, i.e.,
	$(Y_{1}+MI)\cap S=\emptyset,$ a contradiction. Thus $(Y_{1}+Y_{2}I)\cap
	S\neq\emptyset.$
	
	$\Longleftarrow:$ Let $S=M-X$ be a $\phi$-$m$-system. Let $Y$ be a submodule
	of $M$ and $I$ be an ideal of $R$ such that $YI\subseteq X$ and $YI\nsubseteq
	\phi(X).$ Suppose that $Y\nsubseteq X$ and $I\nsubseteq(X:_{R}M).$ Then one
	can see $Y\cap S\neq\emptyset$ and $MI\cap S\neq\emptyset.$ In the definition of $\phi$-$m$-system, consider as
	$Y_{1}=0_{M}$ and $Y_{2}=Y.$ Then since $Y\cap S\neq\emptyset,$ $MI\cap
	S\neq\emptyset$ and $YI\nsubseteq\phi(X)=\phi(S^{c}),$ we obtain $YI\cap
	S=(0_{M}+YI)\cap S\neq\emptyset,$ by $S$ is a $\phi$-$m$-system. Therefore,
	$YI\cap S\neq\emptyset,$ but this contradicts with $YI\subseteq X.$
\end{proof}

\begin{proposition}
	For a proper $X\in S(M),$ let $S:=M-X$. The followings are equivalent:
	
	\begin{enumerate}
		\item $X$ is a $\phi$-prime submodule.
		
		\item If $(Y_{1}+Y_{2})\cap S\neq\emptyset,$ $MI\cap S\neq\emptyset$ and
		$Y_{2}I\nsubseteq\phi(S^{c}),$ for all $Y_{1},Y_{2}\in S(M)$ and any ideal
		$I$ of $R,$ then $(Y_{1}+Y_{2}I)\cap S\neq\emptyset$.
		
		\item If $Y_{2}\cap S\neq\emptyset,$ $MI\cap S\neq\emptyset$ and
		$Y_{2}I\nsubseteq\phi(S^{c}),$ for all $Y_{2}\in S(M)$ and any ideal $I$ of
		$R,$ then $Y_{2}I\cap S\neq\emptyset$.
	\end{enumerate}
\end{proposition}

\begin{proof}
	$(1)\Longrightarrow(2)$ : Assume that $(Y_{1}+Y_{2})\cap S\neq\emptyset,$ $MI\cap S\neq\emptyset$ and
	$Y_{2}I\nsubseteq\phi(S^{c})$ for all $Y_{1},Y_{2}\in S(M)$ and any ideal $I$ of $R.$ Since $X$ is a $\phi$-prime submodule, by Proposition \ref{system}, we know $S=M-X$ is a $\phi$-$m$-system. Also, since $MI\cap S\neq\emptyset,$ $(Y_{1}+MI)\cap S\neq\emptyset.$ Thus, by the definition of $\phi$-$m$-system, $(Y_{1}+Y_{2}I)\cap S\neq\emptyset$.
	
	$(2)\Longrightarrow(3)$ : Set $Y_{1}=0_{M}.$
	
	$(3)\Longrightarrow(1)$ : Suppose that $Y\in S(M)$ and $I$ is an ideal of
	$R$ with $YI\subseteq X$, $YI\nsubseteq\phi(X).$ Let $Y\nsubseteq
	X$ and $I\nsubseteq(X:_{R}M).$ Since $Y\nsubseteq X,$ we have $Y\cap
	S\neq\emptyset.$ Also, as $I\nsubseteq(X:_{R}M),$ i.e., $MI\nsubseteq X$, one
	can see $MI\cap S\neq\emptyset.$ Thus, since $Y\cap S\neq\emptyset, MI\cap
	S\neq\emptyset$ and $YI\nsubseteq\phi(X)=\phi(S^{c}),$ we obtain
	$YI\cap S\neq\emptyset$ by (3). This contradicts with $YI\subseteq X.$
	Hence we are done.
\end{proof}

\begin{definition}
	For $\phi:S(M)\rightarrow S(M)\cup\{\emptyset\},$
	
	\begin{enumerate}
		\item The function $\phi$ is called containment preserving, if for any
		two submodules $X_{1},X_{2}\in S(M),$ $X_{1}\subseteq X_{2}$ implies
		$\phi(X_{1})\subseteq\phi(X_{2}).$
		
		\item The function $\phi$ is called sum preserving, if $\phi(\sum X_{i}%
		)=\sum\phi(X_{i}),$ for all $X_{i}\in S(M).$
	\end{enumerate}
\end{definition}

\begin{lemma}
	Let $\phi$ be containment preserving. Assume that $S\subseteq M$ is a $\phi
	$-$m$-system and $X\in S(M)$ maximal with respect to $X\cap S=\emptyset$ and
	$\phi(X)=\phi(<S^{c}>).$ Then $X$ is a $\phi$-prime submodule of $M$.
\end{lemma}

\begin{proof}
	Let $I$ be any ideal of $R$ and $Y\in S(M)$ such that $YI\subseteq X$ and
	$YI\nsubseteq\phi(X).$ Let $Y\nsubseteq X$ and $I\nsubseteq(X:_{R}M).$ Then as
	$Y\nsubseteq X,$ one can see $X\subsetneq X+Y.$ We show that $(X+Y)\cap
	S\neq\emptyset.$ Indeed, if $(X+Y)\cap S=\emptyset,$ then $X+Y\subseteq
	S^{c},$ so $X+Y\subseteq<S^{c}>$. Thus, $\phi(<S^{c}>)=\phi(X)\subseteq
	\phi(X+Y)\subseteq$ $\phi(<S^{c}>),$ i.e., $\phi(X+Y)=\phi(<S^{c}>).$ This
	doesn't happen because of the properties of $X.$ Also, as $I\nsubseteq
	(X:_{R}M),$ i.e., $MI\nsubseteq X,$ we have $X\subsetneq X+MI.$ We show that
	$(X+MI)\cap S\neq\emptyset.$ Indeed, if $(X+MI)\cap S=\emptyset,$ then similar
	the above, we obtain $\phi(X+MI)=\phi(<S^{c}>)$, a contradiction. Thus, since
	$YI\nsubseteq\phi(X)=\phi(<S^{c}>),(X+Y)\cap S\neq\emptyset$ and $(X+MI)\cap
	S\neq\emptyset,$ one obtains $(X+YI)\cap S\neq\emptyset,$ by $S$ is a $\phi
	$-$m$-system. Then as $YI\subseteq X,$ one gets $X\cap S\neq\emptyset.$ This
	gives us a contradiction. Consequently, one can see that $Y\subseteq X$ or
	$I\subseteq(X:_{R}M)$
\end{proof}

\begin{definition}
	\label{def radical}Let $Y\in S(M).$ If there is a $\phi$-prime submodule $X$
	contains $Y$ such that $\phi(Y)=\phi(X),$ then we define the radical of $Y$ as :
	
	$\sqrt{Y}:=\{x\in M:$ every $\phi$-$m$-system $S$ containing $x$ such that
	$\phi(Y)=\phi(<S^{c}>)$ meets $Y\},$ otherwise $\sqrt{Y}:=M.$
\end{definition}

\begin{theorem}
	\label{the radical}Let $\phi$ be containment and sum preserving. For $Y\in
	S(M),$ let $\Omega:=\{X_{i}\in S(M):X_{i}$ is $\phi$-prime with $Y\subseteq
	X_{i}$ and $\phi(Y)=\phi(X_{i}),$ for $i\in\Lambda$ $\}.$ Then we have
	\[
	\sqrt{Y}=\underset{X_{i}\in\Omega}{\bigcap}X_{i}.
	\]
	
\end{theorem}

\begin{proof}
	Assume that $\sqrt{Y}\neq M.$ Choose $x\in\sqrt{Y}$ and $X_{i}\in\Omega.$ By
	Proposition \ref{system}, we know $S=M-X_{i}$ is a $\phi$-$m$-system. As
	$S\cap Y=\emptyset$ and $x\in\sqrt{Y},$ we have $x\notin S.$ Thus $x\in X_{i}$
	and so $\sqrt{Y}\subseteq\underset{X_{i}\in\Omega}{\bigcap}X_{i}.$ For the
	other containment, choose $y\notin\sqrt{Y}.$ Thus, there is a $\phi$
	-$m$-system $S$ in $M$ with $y\in S,$ $\phi(Y)=\phi(<S^{c}>)$ and $S\cap
	Y=\emptyset.$ Let us consider, the following set :%
	
	\[
	\Delta:=\{X_{i}\in S(M):Y\subseteq X_{i},\text{ }S\cap X_{i}=\emptyset\text{
		and }\phi(X_{i})=\phi(<S^{c}>)\}
	\]
	One can see clearly, $Y\in\Delta,$ so $\Delta\neq\emptyset.$ Let
	$X_{1}\subseteq X_{2}\subseteq\cdot\cdot\cdot\subseteq X_{n}\subseteq
	\cdot\cdot\cdot$ be a chain in $\Delta.$ Then it is easy to see that
	$Y\subseteq\underset{}{\bigcup}X_{i}$ and $S\cap\underset{}{(\bigcup}%
	X_{i})=\emptyset.$ Also, since $\phi$ is containment and sum preserving with
	$\phi(X_{i})=\phi(<S^{c}>),$ one can see $\phi(\underset{}{\bigcup}X_{i}%
	)=\phi(<S^{c}>).$ Thus $\underset{}{\bigcup}X_{i}$ $\in\Delta.$ Hence, by Zorn`s Lemma, $\Delta$ has
	a maximal element, say $X_{i_{1}}.$ Then $y\notin X_{_{i_{1}}},$ since $y\in S$ and $S\cap X_{_{i_{1}}
	}=\emptyset.$ Thus $y\notin$ $\underset{X_{i}\in\Omega}{\bigcap}X_{i},$ so we
	obtain $\underset{X_{i}\in\Omega}{\bigcap}X_{i}\subseteq\sqrt{Y}.$
\end{proof}

\end{document}